\documentclass[10pt,oneside,a4paper]{article}

\usepackage{amscd,amssymb,amsmath,amsthm,mathrsfs,dsfont}
\usepackage[a4paper]{geometry}
\usepackage{color}
\usepackage{tikz} 
\usetikzlibrary{positioning}
\usepackage[shortlabels]{enumitem}
\usepackage{comment}
\usepackage{caption}
\usepackage{subcaption}
 
\usepackage{hyperref}
\hypersetup{
    colorlinks=true,
    linkcolor=blue,
    filecolor=black,      
    urlcolor=blue,
    citecolor=blue,
}
\newtheorem{theorem}{Theorem}
\newtheorem{proposition}{Proposition}
\newtheorem{lemma}{Lemma}
\newtheorem{remark}{Remark}
\newtheorem{definition}{Definition}
\def\N{\mathbb{N}}
\def\Z{\mathbb{Z}}
\numberwithin{equation}{section}
\usepackage{fancyhdr}

\newcommand\shorttitle{Bernoulli field on trees under the removal of isolates}

\begin{document}

\hypersetup{linkcolor=black}
\hypersetup{urlcolor=black}
\title{Gibbs properties of the Bernoulli field on inhomogeneous trees under the removal of isolated sites}
\author{Florian Henning\footnotemark[1] \and Christof Külske\footnotemark[2] \and Niklas Schubert\footnotemark[3]}
\date{\today}

\maketitle
\begin{abstract}
    We consider the i.i.d. Bernoulli field $\mu_p$ with occupation density $p \in (0,1)$ on a possibly non-regular countably infinite tree with bounded degrees. For large $p$, we show that the quasilocal Gibbs property, i.e. compatibility with a suitable quasilocal specification, is lost under the deterministic transformation which removes all isolated ones and replaces them by zeros, while a quasilocal specification does exist at small $p$.

Our results provide an example for an independent field in a spatially non-homogeneous setup which loses the quasilocal Gibbs property under a local deterministic transformation.
\end{abstract}
\textbf{Keywords:} Bernoulli field; non-regular tree; transformed measure; Gibbs property;  quasilocality; cutset.\newline
\noindent\textbf{MSC2020 subject classifications:} 82B26 (primary);
60K35 (secondary).
\footnotetext[1]{Eindhoven University of Technology, Netherlands}
\footnotetext[0]{E-Mail: \href{mailto:f.b.henning@tue.nl}{f.b.henning@tue.nl}; ORCID iD: \href{https://orcid.org/0000-0001-5384-2832}{0000-0001-5384-2832}}
\footnotetext[0]{\url{https://www.researchgate.net/profile/Florian-Henning-2}}
\footnotetext[2]{Ruhr-University Bochum, Germany}
\footnotetext[0]{E-Mail: \href{mailto:Christof.Kuelske@ruhr-uni-bochum.de}{Christof.Kuelske@ruhr-uni-bochum.de}; ORCID iD: \href{https://orcid.org/0000-0001-9975-8329}{0000-0001-9975-8329}}
\footnotetext[0]{\url{https://www.ruhr-uni-bochum.de/ffm/Lehrstuehle/Kuelske/kuelske.html}}
\footnotetext[3]{E-Mail: \href{mailto:Niklas.Schubert@ruhr-uni-bochum.de}{Niklas.Schubert@ruhr-uni-bochum.de}; ORCID iD: \href{https://orcid.org/0009-0000-8912-4701}{0009-0000-8912-4701}}
\footnotetext[0]{\url{https://www.researchgate.net/profile/Niklas-Schubert}}
\tableofcontents
\hypersetup{linkcolor=blue}
\section{Introduction}
The Bernoulli field supported on the countable vertex set of a graph
is the basic object in site percolation theory, see \cite{Gr99}, \cite{BoRi06} and \cite{DuTa16}. By definition of the model, 
each vertex carries an independent Bernoulli variable taking the value $1$ with probability $p\in (0,1)$ and $0$ else. 
One is then interested in connectedness properties of the set of occupied sites, 
in particular the existence and properties of infinite connected clusters,  
in their dependence on the parameter $p$ (the occupation density). 
There is special interest in the concrete cases that the graph is a lattice, or 
a tree, both cases being similar in the very basic sense that at 
small $p<p_c$ there is a.s. no infinite cluster, while there is an infinite cluster for large $p>p_c$. The Bernoulli fields on lattices and trees behave also very differently, in another sense, namely 
in the following fundamental aspect: Whereas there is a.s. uniqueness of the infinite cluster on lattices (see \cite{AiKeNe87}), this is no longer true for regular trees, where the uniqueness of the 
infinite cluster is lost at a second transition value for $p$, see e.g. 
Theorem 8.24 in \cite{LyPe16}. This is an example for the richness of statistical mechanics on trees, and underlines that care is needed when we make predictions from lattice behaviour to tree behaviour and vice versa. 

In the present paper we study the Bernoulli field $\mu_p$ on trees 
under the local transformation $T$ which removes from a random configuration 
all isolated sites, but keeps all clusters of size at least two fixed, including the infinite clusters, 
see below (\ref{eq: Transformation T}) and Figure \ref{fig: Transformation T}.  
Our trees are not assumed to be regular, but the case of regular trees is included. 
We may view the transformation $T$ as a cleansing or straightening-out 
of the configuration of occupied sites where the "dust" of isolated sites is removed. 
This transformation therefore has a flavor of a renormalization group (or coarse-graining) 
transformation in statistical physics, 
where short-range degrees of freedom are integrated out, compare \cite{EFS93} and \cite{FoGiMe18}. For different motivations from stochastic geometry, see the introductions of \cite{JaKu23}
and \cite{EnJaKu22}, and also \cite{Br79} and \cite{Ka05} where thinning processes, such as the discrete Mat\'{e}rn process, are discussed. We then ask for properties of the image measure $\mu_p'$ and we are interested in locality 
properties in the precise sense of 
representability of its conditional probabilities in terms of a quasilocal specification, see below \eqref{def: Quasilocality}.

Studies of local transforms of infinite-volume systems in statistical mechanics have been performed in a number of different geometries, types of
systems and transformations, see \cite{EMSS00}, \cite{RoRu14}, \cite{EAEVIGK12}, \cite{EFHR02}, \cite{FeHoMar14}, \cite{HKLR19}, \cite{KuMe21}, \cite{BeKiKu23}, \cite{AcMaEnAeLe22}, \cite{KR06}, \cite{FHM13}, \cite{HoReZu15}, \cite{BHSJ11}. It has been found that strongly interacting systems 
under local maps may become \textit{not} quasilocally \textit{Gibbsian},
 i.e. become non-representable in terms of quasilocal specifications, while weak interactions 
 tend to lead to Gibbsian behavior.

What to expect from the Bernoulli field on a graph under the projection $T$, in the region 
of large $p$? Typical configurations have very few isolated sites, so 
removing these sites does not seem to change the measure very much. Hence one 
may conjecture that the image field $\mu_p'$ is still nicely behaved, with continuous conditional 
probabilities.  On the contrary it was proved recently \cite{JaKu23} that 
on the integer lattice of dimension at least $2$, $\mu_p'$ is non-Gibbs, see also \cite{EnJaKu22}. This  provides an example of a measure on the lattice which is not just weakly coupled 
but even independent and nonetheless becomes non-Gibbsian
under a strictly local transform (with finite range $1$). The companion measure which arises as the projection to isolates (discrete Mat\'{e}rn process) was shown to behave rather differently, namely quasilocal Gibbs for small enough and large enough $p$, see \cite{EnJaKu22}.

 In our present work we ask whether trees and lattices behave the same or we may see differences, as we do when it comes to the uniqueness of the infinite cluster. Our focus is on regimes of large $p$, as these have found to be the singular ones 
on the lattice. Another strong motivation for us is to generalize from a spatially 
homogeneous situation, and study not only regular trees, 
but also allow for possibly inhomogeneous trees as our base spaces.

\subsection*{Main result and techniques}

Our main result is Theorem \ref{thm: NonGibbsGeneral} which states that on trees with bounded degrees, the Bernoulli field under removal of isolates is non-Gibbsian at large enough 
$p$, with a lower bound on the threshold depending on the upper bound of the local degrees. In the opposite small density regime,  it is Gibbsian, i.e it possesses a quasilocal specification. 

Our proof of the interesting part, namely the non-existence of a quasilocal specification,
is based on the two-layer method (see for example \cite{EFS93} and \cite{AcMaEnAeLe22}), which we apply to our setup of inhomogeneous trees. This first step we combine with a detection and an analysis of an internal phase transition on the first-layer and the proof that this phase transition indeed becomes visible on the second-layer, see Section \ref{sec: Two-layer representation} and Figure \ref{fig: Transformation T}. In this way it builds on \cite{JaKu23}, \cite{EnJaKu22}, but develops the new essential tool of type-changing cutsets, see below, to handle the inhomogeneous tree situation.

The first-layer in our case is the independent Bernoulli field on the tree, the second-layer is coupled 
to the first-layer via the deterministic removal transformation, keeping from the first-layer configuration 
the clusters of sizes greater or equal than two on the same tree, see Figure \ref{fig: Transformation T}. 
In order to study properties of the measure on the second-layer, we need 
to study the conditional system on the first-layer, given second-layer configurations and 
their non-local perturbations.

Let us now outline some key points of our proof which are novel and tree-specific and give an idea why they 
do not rely on spatial homogeneity. Note that the proof of an analogous lattice statement of \cite{JaKu23} is based on shifting alternating configurations  
on the lattice. This does not have an analogy on the inhomogeneous tree, as the graph itself lacks any shift-invariance. So our argument has to be new and different.
The particular conditioning on the second-layer we choose to prove 
non-Gibbsianness is the fully empty conditioning, the resulting associated system 
on the first-layer then becomes a model of particles which are conditioned to stay isolated 
(physically speaking: a hardcore gas). For our discontinuity proof we show the
existence of a phase transition for the latter at large density which can be induced by variations of shapes of volumes arbitrarily far away. On the tree, this means more precisely that there are 
two measures whose configurations, up to local fluctuations, typically resemble 
the alternating configurations of \eqref{eq: Ground states}. These can be selected by appropriate 
balls of even or odd radii uniformly in the ball sizes, see Proposition \ref{prop: Wrong occupied site}.
It is important to understand that inhomogeneous degrees do not spoil this selection argument. As the new and essential tree-typical part of the actual proof  
we then analyze energy and entropy of appropriately defined \textit{type-changing cutsets}, see Definition \ref{def: Cutsets} and Figure \ref{fig: type-changing cutset}. We then perform our analysis in terms of the \textit{pushout method}, introduced in Definition \ref{def: Pushout method}, 
which recursively creates all cutsets of a fixed type. The energy of a cutset is defined in terms of a count of the number of net replacements of zeros by ones needed to relate the two alternating configurations in the inside of the cutset, for which we provide a closed expression on the regular tree and suitable bounds on the general tree. Summarizing, our proof shows that it is not percolation in the original Bernoulli field $\mu_p$ which is relevant for the non-quasilocality of the image measure, but rather the hidden phase transition of hardcore particles on the inhomogeneous tree.   

The remainder of the paper is organized as follows. Section \ref{sec: Definitions} provides the basic definitions and the setting of the model. Moreover, we present the precise statement of our main result Theorem \ref{thm: NonGibbsGeneral} on Gibbs properties of the Bernoulli field on trees under the removal of isolated sites. The proof of this statement is split into the Sections \ref{sec: Two-layer representation} and \ref{sec: Counting cutsets and replacements}. In the first part, Section \ref{sec: Two-layer representation}, we relate the (transformed) second-layer model to a suitably constrained first-layer model which is conditioned on isolation of spins. In the second part, Section \ref{sec: Counting cutsets and replacements}, we finally provide the selection argument for the two distinct groundstates of the first-layer constrained model and prove Proposition \ref{prop: Wrong occupied site}. Here, we develop the pushout method for the analysis of energy and entropy of type-changing cutsets.

\section{Model and main results} \label{sec: Definitions}
Let us state some definitions and constructions, which are needed to understand the work of this paper. This section is based on the constructions in the book of Georgii \cite{Ge11}.

\textbf{Countably infinite trees.} We will investigate random variables indexed by the vertices of a countably infinite tree $(V,E)$ with root $\rho$, where $V$ contains the \textit{vertices} and $E \subset \{e \subset V : |e|=2\}$ the (unoriented) \textit{edges}. If two vertices $x,y \in V$ form an edge $\{x,y\}\in E$, they are called \textit{nearest neighbors} and we write $x \sim y$. For any $\Lambda \subset V$ we set $\partial \Lambda:=\{y\in V \setminus \Lambda: y\sim x, x\in \Lambda\}$ and $\Bar{\Lambda}:= \Lambda \cup \partial \Lambda$. As usual the distance $d(x,y)$ between two vertices $x,y\in V$ is defined by the length of the unique shortest path from $x$ to $y$. Let $W \subset V$, then $(W,E_W)$ can be regarded as a subgraph in the sense that $E_W:=\big\{\{x,y\}\in E : x,y \in W\big\}$.

\textbf{Spin configurations.}
A \textit{spin configuration} is a map assigning to each vertex $x\in V$ a value $\omega_x \in \{0,1\}=:\Omega_0$ and we write $\omega=(\omega_x)_{x \in V}$. Thus, the \textit{configuration space} is defined as $\Omega:=\{0,1\}^V=\big\{\omega=(\omega_x)_{x \in V}: \omega_x \in \{0,1\}~\forall x \in V \big\}$, with the underlying product $\sigma$-algebra $\mathscr{F}:=\big(\mathcal{P}(\Omega_0)\big)^{\otimes V}$ generated by the \textit{spin projections} $\sigma_x:\Omega \xrightarrow{} \Omega_0,~~\omega \mapsto \omega_x, \quad x \in V$. If $\sigma_x=0$, we say the vertex $x$ is \textit{unoccupied} and if $\sigma_x=1$, it is \textit{occupied}. Let $p \in (0,1)$ and $\mu_p:=\text{Ber}(p)^{\otimes V}$ denote the Bernoulli-$p$ product measure on $(\Omega,\mathcal{F})$. Then the process $(\sigma_x)_{x \in V}$ is called the \textit{Bernoulli-$p$ field} on $(V,E)$.

Let us introduce some further notations. First of all, $\Omega_\Lambda:=\{0,1\}^\Lambda$ is the restriction of the configuration space on the subset $\Lambda \subset V$. Given $\Lambda \subset V$, the map $\sigma_\Lambda: \Omega \xrightarrow{} \Omega_\Lambda$ defined by $\omega \mapsto \omega_\Lambda:=(\omega_x)_{x \in \Lambda}$ denotes the projection onto the coordinates in $\Lambda$. Let $\Lambda \subset \Delta \subset V$, $\omega\in \Omega_\Lambda$ and $\eta \in \Omega_{\Delta \setminus \Lambda}$, then the \textit{concatenation} $\omega \eta \in \Omega_\Delta$ is defined by $\sigma_\Lambda(\omega \eta)=\omega$ and $\sigma_{\Delta \setminus \Lambda}(\omega \eta)=\eta$. We will consider events depending only on spins in a certain subset. Therefore, it is useful to define for $\Delta \subset V$ $\mathscr{F}_\Delta:=\sigma\big(\sigma_x,~ x\in \Delta\big)$,
the $\sigma$-algebra on $\Omega$ generated by all events occurring in $\Delta$.

\textbf{Quasilocal specifications.} We are interested in analyzing constraints for the Bernoulli field on the tree with bounded degrees. In order to describe the behaviour of these constraints on the model, we will need the notion of \textit{specifications}. These are families $\gamma=(\gamma_\Lambda)_{\Lambda \Subset V}$ of proper probability kernels each from $(\Omega,\mathscr{F}_{\Lambda^c})$ to $(\Omega, \mathscr{F})$ satisfying a consistency relation. A kernel $\gamma_\Lambda$ for $\Lambda \Subset V$ is said to be \textit{proper} if $\gamma_\Lambda(A|\omega)=\mathds{1}_A (\omega)$ for all $A \in \mathscr{F}_{\Lambda^c}$. Two kernels $\gamma_\Lambda$ and $\gamma_\Delta$ with $\Lambda \subset \Delta \Subset V$ should be compatible in the sense that the following \textit{consistency relation} $(\gamma_\Delta \gamma_\Lambda)(A|\omega)=\gamma_\Delta (A|\omega)$ holds for all $A \in \mathscr{F}$ and $\omega \in \Omega$.

A \textit{local} function is a $\mathscr{F}_\Lambda$-measurable function $f: \Omega \rightarrow \mathbb{R}$ for a $\Lambda \Subset V$. 
    Then, a specification $\gamma=(\gamma_\Lambda)_{\Lambda \Subset V}$ is called \textit{quasilocal} if for each $\Lambda \Subset V$ and every local function $f: \Omega \rightarrow \mathbb{R}$ the following holds
\begin{equation}\label{def: Quasilocality}
    \lim_{n \rightarrow \infty} \sup_{\substack{\zeta, \eta \in \Omega\\ \zeta_{\Lambda_n}=\eta_{\Lambda_n}}} \bigg|\int f(\omega)\gamma_\Lambda(d\omega|\zeta)-\int f(\Tilde{\omega})\gamma_\Lambda(d\Tilde{\omega}|\eta)\bigg|=0,
\end{equation}
where $\big(\Lambda_n\big)_{n \in \mathbb{N}}$ is a \textit{cofinal} sequence, i.e. $\Lambda_n \subset \Lambda_m \Subset V$ for all $n \leq m$ and $\bigcup_{n=1}^\infty \Lambda_n=V$.

\textbf{Gibbs measures and the quasilocal Gibbs property.} 
Let $\gamma=(\gamma_\Lambda)_{\Lambda \Subset V}$ be a specification and $\mu \in \mathscr{M}_1(\Omega, \mathscr{F})$ a probability measure on the infinite volume. We call $\mu$ a \textit{Gibbs measure} for the specification $\gamma$ if it satisfies the \textit{DLR-equation}
\begin{equation*}
    \mu(A| \mathscr{F}_{\Lambda^c})=\gamma_\Lambda(A|\cdot) ~~~~~~~ \mu-\text{almost surely}
\end{equation*}
for all $\Lambda \Subset V$ and $A \in \mathscr{F}$. The set of all Gibbs measures for $\gamma$ is denoted by $\mathscr{G}(\gamma)$. $\mu$ is called \textit{quasilocally Gibbs} if there is a quasilocal specification $\gamma$ such that $\mu \in \mathscr{G}(\gamma)$.

\textbf{Projection to non-isolation.} We are interested whether this property is preserved under the projection to the non-isolated spins. This is the deterministic map $T: \Omega \rightarrow \Omega$ removing the isolated spins of a configuration and is visualized in Figure \ref{fig: Transformation T}. Note that all images in this paper are drawn for regular trees, while our analysis is more general. In more detail, for a configuration $\omega \in \Omega$, the map $T$ is given in a vertex $x \in V$ as
    \begin{equation}\label{eq: Transformation T}
(T\omega)_x:= \omega'_x:=\omega_x \bigg( 1- \prod_{y \in \partial x}(1- \omega_y)\bigg).
\end{equation}

\begin{figure}[t]
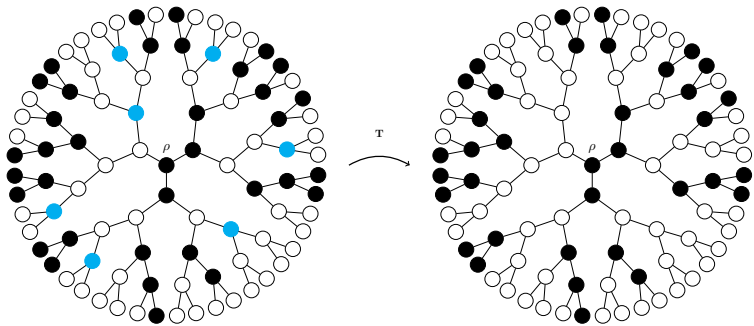

\centering
\scalebox{0.4}{

}
\small \caption{An example of a spin configuration on the binary tree and the application of the map T. Every coloured dot is an occupied site and every uncoloured dot is unoccupied. The blue coloured dots mark the isolated occupied sites.}
\label{fig: Transformation T}
\end{figure}

The main result of this paper is the following Theorem \ref{thm: NonGibbsGeneral}, which states loss of the Gibbs property of the Bernoulli field $\mu_p$ under the transformation $T$ at sufficiently large $p$. Let $\Omega':=T(\Omega) \subset \Omega$ be the image of $T$ and for any $\Lambda \subset V$ let $\mathscr{F}'_\Lambda:= \mathscr{F}_\Lambda \cap \Omega'$.
Consider the \textit{second layer measure} $\mu_p'$ on $(\Omega',  \mathscr{F}'_V)$ defined by 
\begin{equation}
\mu_p':= \mu_p \circ T^{-1}.
\end{equation}

\begin{theorem}\label{thm: NonGibbsGeneral}
	Consider a countably infinite tree $(V,E)$ which is bounded in the sense that there is a number $2 \leq d_{\text{max}}<\infty$ such that each vertex has at least 3 and at most $d_{\text{max}}+1$ nearest neighbours.
	Then, there exist $0<p_1(d_{\text{max}})<p_2(d_{\text{max}})<1$ such that for the independent Bernoulli field $\mu_p$ the following holds true:
	\begin{enumerate}[a)]
		\item
		For $p \in (0,p_1(d_{\text{max}}))$, the transformed measure $\mu'_p$ is quasilocally Gibbs.
		\item For $p \in (p_2(d_{\text{max}}),1)$, the measure $\mu'_p$ is not quasilocally Gibbs.
	\end{enumerate}
\end{theorem}

\begin{remark}
Answering a question of a referee, we conjecture the theorem could hold more generally even in situations with a small enough density of vertices which fail our assumptions on uniform minimal and maximal degree. Clearly \textbf{some} assumptions on the growth of the tree are necessary, as the example of the thinned Bernoulli field on the line $\Z$ shows, which we naturally expect it to be quasilocally Gibbs. 
(Nevertheless, preliminary investigation seems to show that also this one-dimensional 
process displays some remarkable fine properties which are worth an in-depth study.) 
  Now, to find natural conditions on the growth rate of the tree implying non-quasilocality 
and provide proofs for it would require serious investigations. This would put another layer of complexity on top of the ideas of the paper in the present form, which we believe should be the subject of further research.
\end{remark}

\section{Part 1 of the proof: Two-layer representation on trees} \label{sec: Two-layer representation}

The above theorem extends the results obtained in \cite{JaKu23} regarding the situation on the lattice $\Z^d$ with $d \geq 2$ to the geometry of trees with bounded degrees. In particular, it also covers (tree) graphs which are not regular. Note that the proof of an analogous lattice statement of \cite{JaKu23} is based on shifting alternating configurations  
on the lattice, which does not have an analogy on the inhomogeneous tree, so our argument is different. Part a) of Theorem \ref{thm: NonGibbsGeneral} follows from a direct adaption of the argumentation used in \cite{JaKu23} based on Dobrushin-uniqueness theory. In contrast to this, the proof of the more involved part b), which will be given below, takes into account the specific geometry of the tree.

Our goal is to show that there is no quasilocal specification $\gamma'$ for the image measure $\mu'_p$. In order to prove this statement, we will use the two-layer approach. The symbols describing second-layer quantities will carry a prime, to distinguish them from the objects in the first-layer. We will investigate the all-zero configuration $0' \in \Omega'$ and show that this configuration is an essential discontinuity for any quasilocal specification $\gamma'$ which is compatible with the image measure $\mu_p'$. For the purpose of proving this, consider the \textit{constrained first-layer model} $T^{-1}(0')$ supported on isolated configurations. It exhibits two distinct \textit{alternating} ground states $\omega^0$ and $\omega^1$ which are defined as 
\begin{equation}\label{eq: Ground states}
    (\omega^0)_x:=\begin{cases}
    0~~\text{if}~d(\rho,x)~\text{is even,}\\
    1~~\text{if}~d(\rho,x)~\text{is odd}
    \end{cases}
    ~~\text{and}~~~~~
    (\omega^1)_x:=\begin{cases}
    1~~\text{if}~d(\rho,x)~\text{is even,}\\
    0~~\text{if}~d(\rho,x)~\text{is odd.}
    \end{cases}
\end{equation}
They can be transformed into each other by flipping all spins (i.e. changing each unoccupied vertex to an occupied one and vice versa) of the configuration.

Each of these ground states can be evoked by a fitting boundary condition. This is the ball $B_{R}(\rho):=\{x \in V : d(x,\rho) \leq R\}$ around the root with even (for $\omega^0$) or odd (for $\omega^1$) radius $R \in \mathbb{N}_0$ and a fully occupied configuration outside of these balls. In the following, we will abbreviate $B_R(\rho)$ by $B_R$. Let us define a probability measure for the constrained first-layer model for these types of boundary conditions:
\begin{equation*}
\nu_{B_R}(\omega_{B_R}):=\mu_p\big(\sigma_{B_R}=\omega_{B_R}\big|T_{B_R}(\sigma_{B_R}1_{B_R^c})=0'_{B_R}\big).
\end{equation*}
This is the measure conditioned on isolation on $B_R$ and fully occupied boundary condition $1_{B_R^c}$ outside of $B_R$.

Now we can prove, with the following proposition, that the two alternating configurations lead to a phase transition in the first-layer model constrained on isolation.
\begin{proposition}[Phase transition first-layer model]\label{prop: Wrong occupied site}
For the type-1 balls $B_{2R+1}$, the following inequality holds for all $x\in B_2$:
\begin{equation*}
    \sup_{R \in \mathbb{N}_{\geq 2}}\nu_{B_{2R+1}}(\sigma_x\neq \omega_x^1) \leq \epsilon(p),~\text{with}~\lim_{p \uparrow 1} \epsilon(p)=0.
\end{equation*}
The similar statement holds for the type-0 balls $B_{2R}$.
\end{proposition}

This proposition states that the spins in $B_2$ keep some of the information from the boundary with a long-range dependence. We will postpone the proof to Subsection \ref{sec: Counting cutsets and replacements} and continue with the proof of Theorem \ref{thm: NonGibbsGeneral}. In order to use the result of Proposition \ref{prop: Wrong occupied site}, we need to relate the first-layer measure to the second-layer conditional probabilities. For this purpose, consider the configuration $\omega^{\prime *}:=1'_{B_1}0'_{B_1^c} \in \Omega'$ (see Figure \ref{fig: Configuration omega*}). Then,

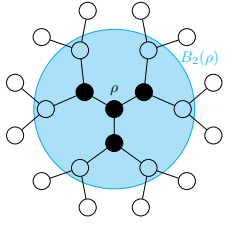
\begin{figure}[t]
\centering
\scalebox{0.45}{
\begin{tikzpicture}[every label/.append style={scale=1.3}]

    \node[shape=circle, draw=cyan, fill=cyan!30, minimum size=4.7cm] (K) at (0,0) {};
    
    \node[shape=circle,draw=black, fill=black, minimum size=0.5cm, label=above:{$\rho$}] (A) at (0,0) {};

    \node[shape=circle, minimum size=0.6cm, label=:{\color{cyan!70}$B_2(\rho)$}] at (2.5,0.8) {};

    \node[shape=circle,draw=black, fill=black, minimum size=0.5cm] (B) at (0,-1) {};
    \node[shape=circle,draw=black, fill=black, minimum size=0.5cm] (C) at (0.866,0.5) {};
    \node[shape=circle,draw=black, fill=black, minimum size=0.5cm] (D) at (-0.866,0.5) {};

    \node[shape=circle,draw=black,minimum size=0.5cm] (BA) at (1,-1.7321) {};
    \node[shape=circle,draw=black,minimum size=0.5cm] (BB) at (-1,-1.7321) {};

    \node[shape=circle,draw=black,minimum size=0.5cm] (BAA) at (2.1213,-2.1214) {};
    \node[shape=circle,draw=black,minimum size=0.5cm] (BAB) at (0.7764,-2.8978) {};

    \node[shape=circle,draw=black,minimum size=0.5cm] (BBA) at (-0.7764,-2.8978) {};
    \node[shape=circle,draw=black,minimum size=0.5cm] (BBB) at (-2.1213,-2.1214) {};

    \node[shape=circle,draw=black, minimum size=0.5cm] (CA) at (2,0) {};
    \node[shape=circle,draw=black, minimum size=0.5cm] (CB) at (1,1.732) {};

    \node[shape=circle,draw=black, minimum size=0.5cm] (CAA) at (2.8978,0.7765) {};
    \node[shape=circle,draw=black, minimum size=0.5cm] (CAB) at (2.8978,-0.7765) {};

    \node[shape=circle,draw=black, minimum size=0.5cm] (CBA) at (0.7765,2.8977) {};
    \node[shape=circle,draw=black, minimum size=0.5cm] (CBB) at (2.1213,2.1212) {};

    \node[shape=circle,draw=black, minimum size=0.5cm] (DA) at (-2,0) {};
    \node[shape=circle,draw=black, minimum size=0.5cm] (DB) at (-1,1.732) {};

    \node[shape=circle,draw=black, minimum size=0.5cm] (DAA) at (-2.8978,-0.7765) {};
    \node[shape=circle,draw=black, minimum size=0.5cm] (DAB) at (-2.8978,0.7765) {};

    \node[shape=circle,draw=black, minimum size=0.5cm] (DBA) at (-2.1213,2.1212) {};
    \node[shape=circle,draw=black, minimum size=0.5cm] (DBB) at (-0.7765,2.8977) {};

    \draw [-] (A) -- (B);
    \draw [-] (A) -- (C);
    \draw [-] (A) -- (D);

    \draw [-] (B) -- (BA);
    \draw [-] (B) -- (BB);

    \draw [-] (BA) -- (BAA);
    \draw [-] (BA) -- (BAB);

    \draw [-] (BB) -- (BBA);
    \draw [-] (BB) -- (BBB);

    \draw [-] (C) -- (CA);
    \draw [-] (C) -- (CB);

    \draw [-] (CA) -- (CAA);
    \draw [-] (CA) -- (CAB);

    \draw [-] (CB) -- (CBA);
    \draw [-] (CB) -- (CBB);

    \draw [-] (D) -- (DA);
    \draw [-] (D) -- (DB);

    \draw [-] (DA) -- (DAA);
    \draw [-] (DA) -- (DAB);

    \draw [-] (DB) -- (DBA);
    \draw [-] (DB) -- (DBB);

    \end{tikzpicture}
    }
    \caption{The configuration $\omega^*\in \Omega$ on the binary tree.}
    \label{fig: Configuration omega*}
    \end{figure}

\begin{lemma}[Relation between the first- and second-layer measure] \label{lemma: relation between first and second layer measure}
    Let $R\in \mathbb{N}$, then:
\begin{equation}\label{eq: relation first and second layer model}
    \frac{\mu_p'(\sigma'_{B_2}=\omega'^*_{B_2}~|~0'_{B_{R+1} \setminus B_2}1'_{B_{R+2} \setminus B_{R+1}})}{\mu_p'(\sigma'_{B_2}=0'_{B_2}~|~0'_{B_{R+1}\setminus B_2} 1'_{B_{R+2} \setminus B_{R+1}})}=\frac{p}{1-p}\nu_{B_{R+1}}(\sigma_{B_2}=\omega^0_{B_2}).
\end{equation}
\end{lemma}

The proof of this lemma can be found in \cite{JaKu23}. Combining Proposition \ref{prop: Wrong occupied site} and Lemma \ref{lemma: relation between first and second layer measure} then allows us to prove Theorem \ref{thm: NonGibbsGeneral}. To ease readability, we will omit the projections in the notation, e.g. will abbreviate expressions like $\mu_p'(\sigma'_\Lambda=\omega'_\Lambda)$ by $\mu_p'(\omega'_\Lambda)$, unless the projections are necessary for understanding.

Assume, there is a specification $\gamma'$ for the image measure $\mu_p'$. Let us relate the conditional probabilities on the left side of equality (\ref{eq: relation first and second layer model}) to the specification $\gamma'$. From the DLR-equation, we obtain the two statements:
\begin{equation}\label{ineq: a_R, b_R}
\begin{split}
     \mu_p'(\omega'_{B_2}|~0'_{B_{2R+1}\setminus B_2}1'_{B_{2R+2}\setminus B_{2R+1}}) &\geq \inf_{\omega'_{(B_{2R})^c}}\gamma'_{B_2}(\omega'_{B_2}|~0'_{B_{2R} \setminus B_2} \omega'_{(B_{2R})^c})=: a_R(\omega'_{B_2})\\
     \mu_p'(\omega'_{B_2}|~0'_{B_{2R}\setminus B_2}1'_{B_{2R+2}\setminus B_{2R}}) &\leq \sup_{\omega'_{(B_{2R})^c}} \gamma'_{B_2}(\omega'_{B_2}|~0'_{B_{2R} \setminus B_2} \omega'_{(B_{2R})^c})=: b_R(\omega'_{B_2}).
     \end{split}
\end{equation}
If $\gamma'$ were quasilocal, we would have $\vert a_R(\omega'_{B_2})- b_R(\omega'_{B_2}) \vert \stackrel{R \rightarrow \infty}{\rightarrow} 0$. This would imply, together with Remark \ref{rk: boundedfrombelow} below that
\begin{equation}\label{eq: quotconv}
     \frac{a_R(\omega_{B_2}^{\prime *})}{b_R(\omega_{B_2}^{\prime *})}
     \xrightarrow[]{R \uparrow \infty} 1
\quad \text{and} \quad    \frac{a_R(0_{B_2}^{\prime})}{b_R(0_{B_2}^{\prime})}
\xrightarrow[]{R \uparrow \infty} 1.     
\end{equation}
Now, consider the right side of (\ref{eq: relation first and second layer model}). The proposition implies
\begin{equation}\label{ineq: first layer type-0}
\nu_{B_{2R}}(\sigma_{B_2}=\omega^0_{B_2})=1-\nu_{B_{2R}}(\sigma_{B_2} \neq \omega^0_{B_2}) \geq 1-|B_2|\epsilon(p)
\end{equation}
and similarly $\nu_{B_{2R+1}}(\sigma_{B_2}=\omega^0_{B_2}) \leq \nu_{B_{2R+1}}(\sigma_\rho=0)\leq \epsilon(p)$. Combining this with (\ref{ineq: a_R, b_R}) and (\ref{ineq: first layer type-0}) leads to
\begin{equation*}
    \frac{\epsilon(p)}{1-|B_{2}|\epsilon(p)}\geq \frac{\nu_{B_{2R+1}}(\sigma_{B_2}=\omega^0_{B_2})}{\nu_{B_{2R}}(\sigma_{B_2}=\omega^0_{B_2})} \geq \frac{a_R(\omega^{\prime *}_{B_2})}{b_R(\omega^{\prime *}_{B_2})}\frac{a_R(0^{\prime}_{B_2})}{b_R(0^{\prime}_{B_2})}. 
\end{equation*}
Considering $R \rightarrow \infty$ for $p \in (0,1)$ sufficiently large (i.e. $\epsilon(p)$ sufficiently small) implies that \eqref{eq: quotconv} is not satisfied.
Hence, a specification for $\mu_p'$ can not be quasilocal. This finishes the proof of Theorem \ref{thm: NonGibbsGeneral}.

\begin{remark}\label{rk: boundedfrombelow}
Let $\omega'_{B_2}=\omega^{\prime *}_{B_2}$ or $\omega'_{B_2}=0'_{B_2}$. Then, for any $p \in (0,1)$ there is a positive constant $c(p)$ such that for all $R \in \{2,3,\ldots\}$ we have $b_R(\omega'_{B_2}) \geq c(p)$.
\end{remark}
\hypersetup{linkcolor=black}The proof of this remark is presented in the \hyperref[sec: Appendix]{Appendix}. It can be adapted to the respective statement for the lattice $\mathbb{Z}^d$ with lattice dimension $d \geq 2$, which proof was not explicitly given in \cite{JaKu23}. For this, consider the observation window $B_3\subset \mathbb{Z}^d$ instead of $B_2\subset \mathbb{Z}^d$. \hypersetup{linkcolor=blue}

\hypersetup{linkcolor=black}
\section{Part 2 of the proof: Cutsets and pushout method} \label{sec: Counting cutsets and replacements}

\hypersetup{linkcolor=blue}It remains to prove the Proposition \ref{prop: Wrong occupied site}. The proof is based on flipping spins in the interior of certain volumes surrounding the vertices in $B_2$. For this purpose, we will construct cutsets with a type and introduce the pushout method.

 We will confine the proof to the type-1 boundary condition. The proof for the type-0 boundary condition proceeds analogously. We need to upper bound the probability $\nu_{B_{2R+1}}(\omega_\rho=0)$ to see a value in the root which is different from that of the preferred ground state $\omega^1$. Note that similar argumentation will give the same upper bound for the other vertices in $B_2$. Consider an arbitrary path from the origin to the boundary. Regarding the isolation constraint in the first-layer, we can not insert an alternating pattern along the path starting with a zero in the root $\rho$. Therefore, there needs to be at least one pair of unoccupied nearest neighbor vertices in the path from the root to the boundary. Among this set, containing the pairs of zeros, we are interested in such pairs minimizing the distance to the root. This leads to the following definition of a type-changing cutset, which is illustrated in Figure \ref{fig: type-changing cutset}.

\begin{definition}\label{def: Cutsets}
Consider a tree $(V,E)$ with bounded degrees and a root $\rho \in V$. The \underline{children} of a vertex $x$ are the $d_x$ nearest neighbours of $x$ which are farther away from $\rho$ than $x$. Fix the orientation pointing away from the root. In more detail, the directed edges are given by the set $ \Vec{E} := \{ \langle x,y \rangle : \{x,y\}\in E,~d(\rho,y)=d(\rho,x)+1\}$.
  \begin{enumerate}[a)]
\item For $\emptyset \neq \Lambda \subsetneq V$, we call $\vec{L}(\Lambda):=\{ \langle x,y \rangle \in \vec{E} : x \in \Lambda,y \in V \setminus \Lambda \}$ the \underline{cutset} for $\Lambda$.
 \item A subset $\vec{L} \subset \vec{E}$ is called a \underline{type-changing cutset of type $0/1$} iff it is the cutset for a finite subtree $\Lambda \Subset V$ with root $\rho$ such that every leaf of $\Lambda$ has an \textbf{even/odd} distance to the root. In both cases, we call the vertices of the subtree $\Lambda$ the \underline{interior} of the cutset $\vec{L}$, denoted by $int(\vec{L})$.
 For $s=0,1$, we set
 \begin{equation*}
  \mathscr{C}(s,V):=\{\vec{L} \subset \vec{E}~:~\vec{L}~\text{type-changing cutset of type $s$}\}.
 \end{equation*} 
 \item Let us denote the \underline{boundary} of a cutset $\Vec{L} \subset \Vec{E}$ with $\partial\Vec{L}:=\{y\in V: \langle x,y \rangle \in \Vec{L}\}\subset V$ and the \underline{closure} as $cl(\Vec{L}):= int(\Vec{L})\cup \partial \Vec{L} \subset V$. 
 \item  A configuration $\omega \in T^{-1}(0')$ is \underline{adapted} to a type-changing cutset $\vec{L}$ of type $s$ if $\omega_{int(\Vec{L})}=\omega_{int(\Vec{L})}^s$, where $\omega^s$ is the the groundstate of type $s$ (see (\ref{eq: Ground states})) and $\omega_{\partial{\Vec{L}}}=0_{\partial{\Vec{L}}}$. 
 \end{enumerate}
\end{definition}

\begin{remark}
 Note that our type-changing cutsets should not be confused with the \textit{tree-contours} which were used in \cite{GaRuSh12}, \cite{CoKuLN22} whose geometric part consists of subtrees.

In particular, let $\Vec{L}$ be a type-changing cutset and $\omega\in T^{-1}(0')$ a configuration being adapted to $\Vec{L}$. Note that $\Vec{L}$ specifies the configuration in the closure $cl(\vec{L})$, while there may be many adapted configurations which differ on the outside. This distinguishes the notion of being adapted to a cutset from that of being compatible to a contour in the Peierls argument.
\end{remark}

An example of a type-changing cutset of type 0 is illustrated in Figure \ref{fig: type-changing cutset}. Concerning the considerations made before, every different-valued spin inside of the ball $B_{2R+1}$ with respect to the preferred ground state $\omega^1$ has to be surrounded by a type-changing cutset $\Vec{L}\in \mathscr{C}(0,B_{2R+1})$. Therefore, we obtain
\begin{align}
    \nonumber \nu_{B_{2R+1}}(\omega_\rho=0) \leq \sum_{\Vec{L}\in \mathscr{C}(0,B_{2R+1})} \nu_{B_{2R+1}}\big(\omega:~ \omega~\text{is adapted to}~\Vec{L}\big).
\end{align}
Now, let us take a look at each term of the sum. Let $\Vec{L} \in \mathscr{C}(0,B_{2R+1})$ be a type-changing cutset of type 0, then we have
\begin{equation}\label{eq: Decomposition probability Peierls}
    \nu_{B_{2R+1}}\big(\omega:~ \omega~\text{is adapted to}~\Vec{L}\big)=\frac{W(\omega^0_{int(\Vec{L})})(1-p)^{|\partial \Vec{L}|}Z_{B_{2R+1} \setminus cl(\Vec{L})}}{Z_{B_{2R+1}}},
\end{equation}
where $W(\omega_\Lambda):=\prod_{x \in \Lambda} (1-p)^{1-\omega_x}p^{\omega_x}$ are the Bernoulli weights of a subset $\Lambda \Subset V$. Moreover, $Z_\Lambda$ is the partition function over all configurations in $\Lambda \subseteq B_{2R+1}$ being compatible with the boundary condition outside of $B_{2R+1}$ under the isolation constraint. The idea is to flip all the spins of the configuration $\omega$ on $int(\Vec{L})$. For large $p$, this will lead to an energetically more favourable configuration, since we flip more unoccupied vertices to occupied vertices. The net replacements inside of $\Vec{L}$ are given by
\begin{equation}
    N^{\Vec{L}}_{\text{repl}}:=|\{x \in int(\Vec{L}):~\omega^0_x=0\}|-|\{y \in int(\Vec{L}):~\omega^0_y=1\}|.
\end{equation}

\begin{figure}[t]
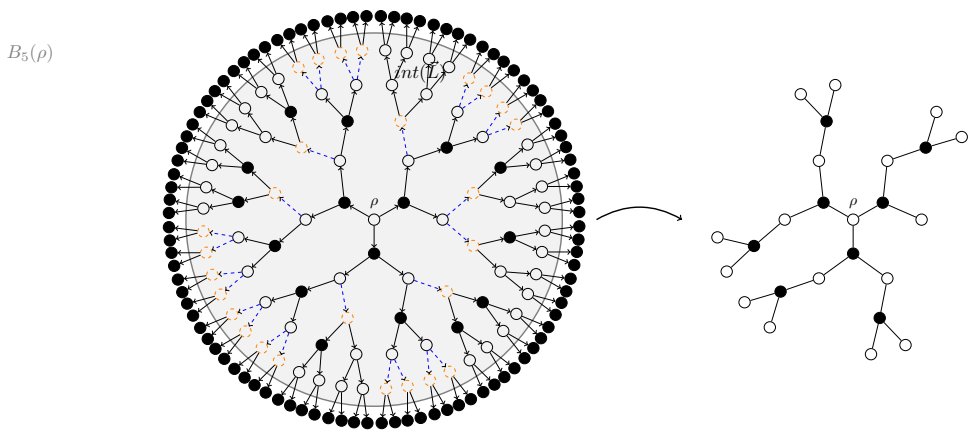

    \centering
\scalebox{0.45}{

}
    \caption{An illustration of the type-1 boundary condition $B_5(\rho)$ on the binary tree together with an isolated configuration inside of the ball. The pictured configuration is unoccupied at the root, hence it is adapted to a type-changing cutset $\vec{L} \in \mathscr{C}(0,B_5(\rho))$ of type $0$, which cuts off the rooted subtree, where the configuration resembles the groundstate $\omega^0$, from the outside $B_5(\rho) \setminus int(\Vec{L})$. The cutset edges are dashed and coloured in blue and the boundary $\partial \Vec{L}$ is dashed and orange. The interior $int(\vec{L})$ of the cutset is pictured on the right side.}
    \label{fig: type-changing cutset}
\end{figure}

Furthermore, the flipping results in an allowed configuration for the model constrained on isolation, because $int(\Vec{L})$ is surrounded by a layer of zeros and we can lower bound the partition function as follows
\begin{equation*}
    Z_{B_{2R+1}} \geq W(\omega^1_{int(\Vec{L})})(1-p)^{|\partial\Vec{L}|}Z_{B_{2R+1} \setminus cl(\Vec{L})}.
\end{equation*} 
Consequently, we obtain that the l.h.s. of \eqref{eq: Decomposition probability Peierls} is bounded from above by $\big(\frac{1-p}{p}\big)^{N^{\Vec{L}}_{\text{repl}}}$. Thus,
\begin{equation}\label{ineq: First Peierls estimation}
    \nu_{B_{2R+1}}(\omega_\rho=0)
\leq \sum_{\Vec{L}\in \mathscr{C}(0,B_{2R+1})} \big(\frac{1-p}{p}\big)^{N^{\Vec{L}}_{\text{repl}}}
\end{equation}
and it remains to determine a specific expression for the replacements of these cutsets. For the purpose of bounding the number of cutsets having a specific number of replacements, we will relate the replacements of each cutset with the number of vertices in the interior. By the assumption of bounded degrees of the tree, the number of possible cutsets with a given number $n$ of vertices in the interior then growths at most exponentially fast with $n$ (see Lemma \ref{lemma: Entropy lemma} below).

\subsection*{Pushout method.} 
 In order to relate the number of net replacements $N^{\Vec{L}}_{\text{repl}}$ in a cutset to the number of vertices $|int(\Vec{L})|$ in the interior, we introduce a pushout method on the tree with bounded degrees. In more detail, let us consider an arbitrary cutset $\Vec{L}$ in $\mathscr{C}(0,V)$ or $\mathscr{C}(1,V)$ and the respective \textit{initial} cutset  $\Vec{L}_0$ with smallest possible interior. Starting from $\Vec{L}_0$, we can obtain $\Vec{L}$ by a unique (up to permutations of the order) sequence of finitely many pushout operations of the cutset edges (see Figure \ref{fig: Construction cutsets}). While pushing out these cutset edges, one can count the number of vertices in the interior $|int(\Vec{L})|$ and the net replacements $N^{\Vec{L}}_{\text{repl}}$.

\begin{definition}\label{def: Pushout method}
Let $\langle x,y \rangle \in \Vec{E}$ be an oriented edge of the tree $(V,E)$ with bounded degrees. Assume that $y$ has $d$ children $z_1, \ldots, z_d$ and further assume that each $z_i$ has $d_{z_i}$ children $v_{i1}, \ldots, v_{id_{z_i}}$. The \underline{pushout operation} applied to the cutset edge $\langle x,y \rangle$ is a map $\pi_{\langle x,y \rangle}: \mathscr{C}(s,V) \rightarrow \mathscr{C}(s,V)$, where $s=0$ if $d(\rho,x)$ is \textbf{even} and $s=1$ if $d(\rho,x)$ is \textbf{odd}. This map is defined as follows, $\pi_{\langle x,y \rangle}(\Vec{L})=\Vec{L}'$ if $\langle x,y \rangle \in \Vec{L}$ and $\pi_{\langle x,y \rangle}(\Vec{L})=\Vec{L}$ otherwise. Here, $\Vec{L}'$ emerges by removing the cutset edge $ \langle x,y \rangle $ from $\Vec{L}$ and replacing it by the $\prod_{i=1}^{d} d_{z_i}$ edges $\langle z_i,v_{ij} \rangle$ (see Figure \ref{fig: Construction cutsets} (b)).
\end{definition}

\begin{figure}
     \begin{subfigure}[b]{0.4\textwidth}
         \scalebox{0.7}{
         \begin{tikzpicture}[every label/.append style={scale=1.1}]

        \scalebox{0.8}{\node[minimum size=0.5cm, label=:{Type 0}] at (1.2,2.4) {};}

        \scalebox{0.8}{\node[minimum size=0.5cm, label=:{Type 1}] at (6.3,2.4) {};}

		\node[shape=circle,draw=black, very thick, minimum size=0.5cm, label=above:{$\rho$}] (center) at (1,0) {};
		\node[shape=circle,dashed, draw=orange, very thick, minimum size=0.5cm] (lacircle) at (-0.0825,0.625) {};
		\node[shape=circle,dashed, draw=orange, very thick, minimum size=0.5cm] (racircle) at (2.0825,0.625) {};
		\node[shape=circle,dashed, draw=orange, very thick, minimum size=0.5cm] (bcircle) at (1,-1.25) {};
		
		\draw[->, very thick, dashed, draw=blue] (center) to (lacircle);
		\draw[->, very thick, dashed, draw=blue] (center) -- (racircle);
		\draw[->, very thick, dashed, draw=blue] (center) -- (bcircle);

  \node[shape=circle,draw=black, fill=black, minimum size=0.5cm, label=above:{$\rho$}] (A) at (5,0) {};
		
		\node[shape=circle,draw=black, very thick,minimum size=0.5cm] (B) at (5,-1) {};
		\node[shape=circle,draw=black,  very thick,minimum size=0.5cm] (C) at (5.866,0.5) {};
		\node[shape=circle,draw=black,  very thick,minimum size=0.5cm] (D) at (4.134,0.5) {};
		
		\node[shape=circle,dashed, draw=orange, very thick,minimum size=0.5cm] (BA) at (6,-1.7321) {};
		\node[shape=circle,dashed, draw=orange,  very thick,minimum size=0.5cm] (BB) at (4,-1.7321) {};
		
		\node[shape=circle,dashed, draw=orange,  very thick,minimum size=0.5cm] (CA) at (7,0) {};
		\node[shape=circle,dashed, draw=orange,  very thick,minimum size=0.5cm] (CB) at (6,1.732) {};
		
		\node[shape=circle,dashed, draw=orange,  very thick,minimum size=0.5cm] (DA) at (3,0) {};
		\node[shape=circle,dashed, draw=orange,  very thick,minimum size=0.5cm] (DB) at (4,1.732) {};

		\draw [->, very thick] (A) -- (B);
		\draw [->, very thick] (A) -- (C);
		\draw [->, very thick] (A) -- (D);

		\draw [->, draw=blue,  dashed, very thick] (B) -- (BA);
		\draw [->, draw=blue, dashed, very thick] (B) -- (BB);
		\draw [->, draw=blue, dashed, very thick] (C) -- (CA);
		\draw [->, draw=blue, dashed, very thick] (C) -- (CB);
		\draw [->, draw=blue, dashed, very thick] (D) -- (DA);
		\draw [->, draw=blue, dashed, very thick] (D) -- (DB);
\end{tikzpicture}
}
         \caption{The initial cutsets}
     \end{subfigure}
     \hfill
     \begin{subfigure}[b]{0.45\textwidth}
     \scalebox{0.7}{ 
\begin{tikzpicture}[every label/.append style={scale=1.1}]
	\node[shape=circle,draw=black, very thick, minimum size=0.5cm, label=above:{$x$}] (A) {};
	\node[shape=circle,dashed, draw=orange, very thick, minimum size=0.5cm, label=above:{$y$}] (B) [right=0.75cm of A] {};
	
	\node[shape=circle,draw=black, very thick, minimum size=0.5cm, label=above:{$x$}] (C) [right=2cm of B] {};
	\node[shape=circle,draw=black, very thick, minimum size=0.5cm, label=above:{$y$}, fill=black] (D) [right=1cm of C] {};
	\node[shape=circle,draw=black, very thick, minimum size=0.5cm, label=above:{$z_1$}] (E) [below right=0.5cm and 0.5cm of D] {};
	\node[shape=circle,draw=black, very thick, minimum size=0.5cm, label=above:{$z_2$}] (F) [above right=0.5cm and 0.5cm of D] {};
	\node[shape=circle,dashed, draw=orange, very thick, minimum size=0.5cm,label=right:{$v_{11}$}] (E_1) [below right=0.10cm and 0.60cm of E] {};
	\node[shape=circle,dashed, draw=orange, very thick, minimum size=0.5cm,label=right:{$v_{12}$}] (E_2) [above right=0.10cm and 0.60cm of E] {};
	\node[shape=circle,dashed, draw=orange, very thick, minimum size=0.5cm,label=right:{$v_{21}$}] (F_1) [below right=0.10cm and 0.60cm of F] {};
	\node[shape=circle,dashed, draw=orange, very thick, minimum size=0.5cm,label=right:{$v_{22}$}] (F_2) [above right=0.10cm and 0.60cm of F] {};
	
	\draw[->, very thick, dashed,draw=blue] (A) -- (B);
 \scalebox{0.85}{\node[minimum size=0.5cm, label=:{Pushout}] at (3,0.4) {};}
	\draw[->, shorten >=5pt, shorten <=5pt, black, very thick] (B) to [bend left] (C);
 \scalebox{0.85}{\node[minimum size=0.5cm, label=:{Merging}] at (3,-1.5) {};}
\draw[->, shorten >=5pt, shorten <=5pt, black, very thick, label=above:{\text{merging}}] (C) to [bend left] (B);
	\draw[->, very thick,draw=black] (C) -- (D);
	\draw[->, very thick] (D) -- (E);
	\draw[->, very thick] (D) -- (F);
	\draw[->, dashed, very thick, draw=blue] (E) -- (E_1);
	\draw[->, dashed, very thick, draw=blue] (E) -- (E_2);
	\draw[->, dashed, very thick, draw=blue] (F) -- (F_1);
	\draw[->, dashed, very thick, draw=blue] (F) -- (F_2);
	\end{tikzpicture}  
 }
         \caption{Pushout and Merging operation}
     \end{subfigure}
        \caption{Algorithmic construction of type-changing cutsets.}
        \label{fig: Construction cutsets}
\end{figure}
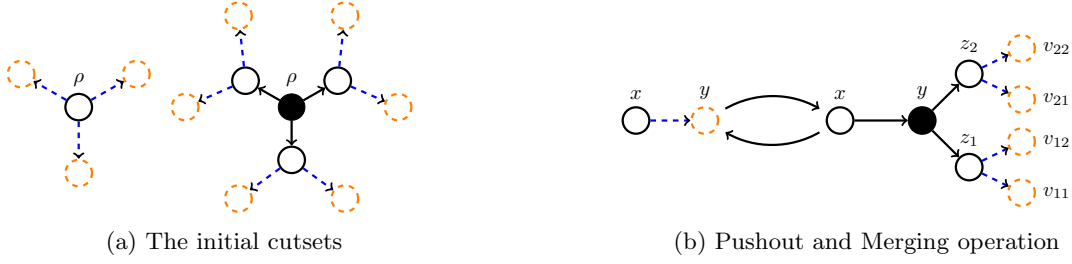

\begin{lemma}\label{lem: PushoutAlgorithm}
Every cutset $\Vec{L} \subset \Vec{E}$ of type $s \in \{0,1\}$ can be obtained from the initial cutset $\Vec{L}_0$ of the respective type by a finite number of pushout operations.
\end{lemma}

\begin{proof}
The result is based on the following algorithm, which is visualized in Figure \ref{fig: Construction cutsets}:
\begin{enumerate}
	\item If $\Vec{L}$ is the initial cutset of type $s$ (see Figure \ref{fig: Construction cutsets} (a)), then terminate. \\Otherwise, let $n:= d(\rho, \partial \Vec{L})=\max_{v\in \partial \Vec{L}} d(\rho,v) \geq 2$. Note that $n$ is odd if $\Vec{L}$ is of type zero and even otherwise. 
	\item Choose any $v \in \partial \Vec{L} \cap \partial B_{n-1}(\rho)$ and let $y$ denote the unique ancestor of $v$ two steps back from $v$ (grandparent). By definition of the cutset together with its fixed type and of the number $n$, all children of $y$ are starting points of cutset edges (see the r.h.s. of Figure \ref{fig: Construction cutsets} (b)).
	\item Perform a merging operation, i.e. replace these cutset edges by the directed edge with end point $y$ (see the l.h.s. of Figure \ref{fig: Construction cutsets} (b)). Substitute $\Vec{L}$ by the so obtained cutset.
	\item Go back to step 1.  
\end{enumerate}	
This algorithm terminates, as finitely many repetitions will decrease the finite number $n$ by steps of size two eventually leading to $\Vec{L}_0$. Substituting the merging operation with the pushout operation and applying these steps in reversed order gives us the statement of the lemma.  
\end{proof}
Performing the pushout operation on an edge of a cutset $\Vec{L}$ alters the relevant quantities, the number $|int(\Vec{L})|$ of vertices in the interior and $N^{\Vec{L}}_{\text{repl}}$ of net replacements of zeros by ones in $int(\Vec{L})$ needed to remove (in terms of flipping the spins) $\Vec{L}$, in the following way:

\begin{remark}\label{rk: Pushout relevant quantities}
If $\vec{L}$ is a type-changing cutset of type $0$ or $1$ and $\langle x,y \rangle\in \vec{L}$ is such that $y$ has $d_y$ children, then for the new cutset $\vec{L}'$ obtained by performing the pushout operation on $\langle x,y \rangle$ the following holds true:
\begin{enumerate}[a)]
	\item  $\vert int(\vec{L}')\vert=\vert int(\vec{L}) \vert+d_y+1$ and   
	\item $N^{\vec{L}'}_{\text{repl}}=N^{\vec{L}}_{\text{repl}}+d_y-1$.
\end{enumerate}
\end{remark}
This leads to the following statement
\begin{lemma}\label{lemma: Properies of cutset}
If $\vec{L}_n$ is a type-changing cutset of type 0 which is obtained from the initial cutset $\vec{L}_0$ by $n$ pushout operations, then we have the following bounds 
\begin{enumerate}[a)]
\item $\vert int(\vec{L}_n) \vert \leq 1+n(d_{\text{max}}+1)$
\item $N_{\text{repl}}^{\vec{L}_n} \geq 1+n(d_{\text{min}}-1)=:r_n$.
\end{enumerate}
For $\vec{L}_n$ a type-changing cutset of type 1, the same statements hold true with $n$ replaced by $n+1$.
\end{lemma}
The proof of this lemma relies on an induction on the number of pushout operations needed to construct a cutset starting with the initial cutsets (see Lemma \ref{lem: PushoutAlgorithm} and Figure \ref{fig: Construction cutsets}).

Resumming in \eqref{ineq: First Peierls estimation} over the number of pushout operations needed to construct a cutset $\vec{L}$ we arrive at:
\begin{equation}\label{eq: Wrong spin root}
\nu_{B_{2R+1}}(\omega_\rho=0) \leq \sum_{n = 0}^\infty \big |\{\vec{L}_n \in \mathscr{C}(0,V) : \vec{L}_n \text{ is obtained by } n \text{ pushouts }\}\big|\big(\frac{1-p}{p}\big)^{r_n}.  
\end{equation}

To bound the combinatorial weight, we apply the following Lemma, which is a well-known result also used e.g. in \cite{GaRuSh12}.
\begin{lemma}\label{lemma: Entropy lemma}
    Let G be a graph of maximal degree $d_{\text{max}}+1$. Then the number of connected subgraphs $\Gamma \subset G$ with $k$ edges, containing a given vertex is bounded from above by
    \begin{equation*}
        (d_{\text{max}}+1)^{2k}.
    \end{equation*}
\end{lemma}
The proof is an immediate extension of a result for the lattice in \cite{FV17} to the setup of a general graph with bounded degrees, which follows from Lemma 3.38 in \cite{FV17}.

\begin{remark}\label{rk: Relation edges and bonds}
Let $W \Subset V$ be a connected subset of $V$. Then the relation between the edges and the vertices of the subgraph reads $|W|=|E_W|+1$.
\end{remark}

Lemmas \ref{lemma: Properies of cutset} and \ref{lemma: Entropy lemma} and Remark \ref{rk: Relation edges and bonds} give for any $n \in \N_0$
\begin{equation}\label{eq: Combinatorial weights}
 |\{\vec{L}_n \in \mathscr{C}(0,V) : \vec{L}_n \text{ is obtained by } n \text{ pushout operations}\}\big| \leq (d_{\text{max}}+1)^{2n(d_{\text{max}}+1)}.
\end{equation}
Combining \eqref{eq: Wrong spin root} and \eqref{eq: Combinatorial weights} with the lower bound in Lemma \ref{lemma: Properies of cutset} b) yields
\begin{equation}\label{eq: cutset final}
\begin{split}
\nu_{B_{2R+1}}(\omega_\rho=0)\leq \sum_{n = 0}^\infty (d_{\text{max}}+1)^{2n(d_{\text{max}}+1)}\big(\frac{1-p}{p}\big)^{1+n(d_{\text{min}}-1)}.
\end{split}
\end{equation}
Note that the right hand side of (\ref{eq: cutset final}) goes to zero for $p \uparrow 1$, which concludes the statement of Proposition \ref{prop: Wrong occupied site}.

\section*{Appendix}\label{sec: Appendix}
\addcontentsline{toc}{section}{Appendix}

\hypersetup{linkcolor=black}

\textbf{Proof of Remark \ref{rk: boundedfrombelow}.}
Let $\omega'_{B_2}=\omega^{\prime *}_{B_2}$ or $\omega'_{B_2}=0'_{B_2}$. 
The definition of $b_R(\omega^{'}_{B_2})$ and Bayes' formula yield \hypersetup{linkcolor=blue}
\begin{equation}\label{eq: UniformBoundforDenom1}
b_R(\omega^{'}_{B_2}) \geq \mu_p'(\omega^{'}_{B_2}|~0'_{B_{2R}\setminus B_2}1'_{B_{2R+2}\setminus B_{2R}})=\frac{1}{1+\frac{\mu_p'(0'_{B_{2R}\setminus B_2}1'_{B_{2R+2}\setminus B_{2R}} |~ (\omega^{'}_{B_2})^c)}{\mu_p'(0'_{B_{2R}\setminus B_2}1'_{B_{2R+2}\setminus B_{2R}} |~ \omega^{'}_{B_2})}\frac{\mu_p'((\omega^{'}_{B_2})^c)}{\mu_p'(\omega^{'}_{B_2})}},
\end{equation}
where $\mu_p'( (\omega^{'}_{B_2})^c):=\mu_p'(\sigma'_{B_2} \neq \omega^{'}_{B_2})$.
It remains to bound the denominator from above. From the definition of the transformed measure $\mu_p'$ it follows 
\begin{equation}\label{eq: UniformBoundforDenom2}
\begin{split}
&\frac{\mu_p'(0'_{B_{2R}\setminus B_2}1'_{B_{2R+2}\setminus B_{2R}} |~ (\omega^{'}_{B_2})^c)}{\mu_p'(0'_{B_{2R}\setminus B_2}1'_{B_{2R+2}\setminus B_{2R}} |~\omega^{'}_{B_2})}\frac{\mu_p'((\omega^{'}_{B_2})^c)}{\mu_p'(\omega^{'}_{B_2})}= \frac{\mu_p'((\omega^{'}_{B_2})^c0'_{B_{2R}\setminus B_2}1'_{B_{2R+2}\setminus B_{2R}} )}{\mu_p'( \omega^{'}_{B_2}0'_{B_{2R}\setminus B_2}1'_{B_{2R+2}\setminus B_{2R}})} \cr
&=\frac{\sum_{\omega_{B_2}} \sum_{\omega_{B_{2R}\setminus B_2}}f(\omega_{B_{2}}\omega_{B_{2R} \setminus B_2})\mathds{1}_{\{T_{B_2}(\omega_{B_2}\omega_{B_{2R}\setminus B_2}) \neq \omega^{'}_{B_2} \}}\mu_p(\omega_{B_2}\omega_{B_{2R}\setminus B_2}1_{B_{2R+2}\setminus B_{2R}})}{\sum_{\tilde{\omega}_{B_2}} \sum_{\tilde{\omega}_{B_{2R}\setminus B_2}}f(\tilde{\omega}_{B_{2}}\tilde{\omega}_{B_{2R} \setminus B_2})\mathds{1}_{\{T_{B_2}(\tilde{\omega}_{B_2}\tilde{\omega}_{B_{2R}\setminus B_2}) = \omega^{'}_{B_2} \}}\mu_p(\tilde{\omega}_{B_2}\tilde{\omega}_{B_{2R}\setminus B_2}1_{B_{2R+2}\setminus B_{2R}})}.
\end{split}
\end{equation}
Here, $f(\omega_{B_{2}}\omega_{B_{2R} \setminus B_2}):=\mathds{1}_{\{T_{B_{2R} \setminus B_2}(\omega_{B_2}\omega_{B_{2R}\setminus B_2}1_{B_{2R+2}\setminus B_{2R}})=0'_{B_{2R} \setminus B_2}\}}$. Recall that $\mu_p=\text{Ber}(p)^{\otimes V}$ and (\ref{eq: UniformBoundforDenom2}) reads
\begin{equation}\label{eq: UniformBoundforDenom3}
    \frac{\sum_{\omega_{B_2}}W(\omega_{B_2}) \sum_{\omega_{B_{2R}\setminus B_2}}W(\omega_{B_{2R} \setminus B_2})f(\omega_{B_{2}}\omega_{B_{2R} \setminus B_2})\mathds{1}_{\{T_{B_2}(\omega_{B_2}\omega_{B_{2R}\setminus B_2}) \neq \omega^{'}_{B_2} \}}}{\sum_{\tilde{\omega}_{B_2}}W(\tilde{\omega}_{B_2})\sum_{\tilde{\omega}_{B_{2R}\setminus B_2}}W(\tilde{\omega}_{B_{2R} \setminus B_2})f(\tilde{\omega}_{B_2}\tilde{\omega}_{B_{2R} \setminus B_2})\mathds{1}_{\{T_{B_2}(\Tilde{\omega}_{B_2}\Tilde{\omega}_{B_{2R}\setminus B_2}) = \omega^{'}_{B_2} \}}},
\end{equation}
where we recall that $W(\omega_\Lambda)=\prod_{x \in \Lambda} (1-p)^{1-\omega_x}p^{\omega_x}$ are the Bernoulli weights of a subset $\Lambda \Subset V$.
Now let $\hat{\omega}_{B_2}$ be any configuration which satisfies $\hat{\omega}_{B_2 \setminus B_1}=0_{B_2 \setminus B_1}$ and $T_{B_2}(\hat{\omega}_{B_2})=\omega_{B_2}'$. In the case $\omega'_{B_2}=\omega^{\prime *}_{B_2}$, the only possible choice is $\hat{\omega}_{B_2}=\omega^{*}_{B_2}$. In the case $\omega'_{B_2}=0^{'}_{B_2}$, we may simply take $\hat{\omega}_{B_2}=0_{B_2}$. Restricting the denominator of \eqref{eq: UniformBoundforDenom3} to the term with $\tilde{\omega}_{B_2}=\hat{\omega}_{B_2}$ provides the upper bound
\begin{equation}\label{eq: UniformBoundforDenom4}
    \frac{\sum_{\omega_{B_2}}W(\omega_{B_2}) \sum_{\omega_{B_{2R}\setminus B_2}}W(\omega_{B_{2R} \setminus B_2})f(\omega_{B_{2}}\omega_{B_{2R} \setminus B_2})\mathds{1}_{\{T_{B_2}(\omega_{B_2}\omega_{B_{2R}\setminus B_2}) \neq \omega^{'}_{B_2} \}}}{W(\hat{\omega}_{B_2})\sum_{\tilde{\omega}_{B_{2R}\setminus B_2}}W(\tilde{\omega}_{B_{2R} \setminus B_2})f(\hat{\omega}_{B_2}\tilde{\omega}_{B_{2R} \setminus B_2})\mathds{1}_{\{T_{B_2}(\hat{\omega}_{B_2}{\Tilde{\omega}_{B_{2R}\setminus B_2}) = \omega^{'}_{B_2} \}}}},
\end{equation}
where we note that in the case $\omega'_{B_2}=\omega^{\prime *}_{B_2}$ this upper bound becomes an equality. The assumption $\hat{\omega}_{B_2 \setminus B_1}=0_{B_2 \setminus B_1}$ guarantees
that the indicator in the denominator is constantly one and the inequality $f(\omega_{B_{2}}\omega_{B_{2R} \setminus B_2}) \leq f(\hat{\omega}_{B_{2}}\omega_{B_{2R} \setminus B_2})$ holds for all $\omega_{B_{2}}$ and $\omega_{B_{2R} \setminus B_2}$. Moreover, we can upper bound the indicator in the numerator by one. Hence, we obtain the upper bound for \eqref{eq: UniformBoundforDenom4}
\begin{equation}\label{eq: UniformBoundforDenom5}
    \frac{\sum_{\omega_{B_2}}W(\omega_{B_2})}{W(\hat{\omega}_{B_2})}= \frac{1}{W(\hat{\omega}_{B_2})}< \infty,
\end{equation}
which is an $R$-independent upper bound.
Combining \eqref{eq: UniformBoundforDenom1}-\eqref{eq: UniformBoundforDenom5} concludes the proof of Remark \ref{rk: boundedfrombelow}.

\hypersetup{urlcolor=blue}

\section*{Acknowledgements}
The authors thank the anonymous referee for insightful comments and suggestions.

\end{document}